\documentclass{amsart}
\usepackage{amsfonts}
\usepackage{amsmath}
\usepackage{amssymb}
\usepackage{graphicx}

\providecommand{\U}[1]{\protect\rule{.1in}{.1in}}
\newtheorem{theorem}{Theorem}
\theoremstyle{plain}
\newtheorem{acknowledgement}{Acknowledgement}

\newtheorem{corollary}{Corollary}

\newtheorem{lemma}{Lemma}

\newtheorem{proposition}{Proposition}

\numberwithin{equation}{section}
\input{tcilatex}

\begin{document}
\title[On van der Corput property of squares]{On van der Corput property of
squares}
\author{Sini\v{s}a Slijep\v{c}evi\'{c}}
\address{Department of Mathematics, Bijeni\v{c}ka 30, Zagreb,\ Croata}
\email{slijepce@math.hr}
\urladdr{}
\date{September 21, 2009}
\subjclass[2000]{Primary 11P99; Secondary 37A45}
\keywords{S\'{a}rk\"{o}zy theorem, recurrence, difference sets, positive
definiteness, van der Corput property, Fourier analysis}

\begin{abstract}
We prove that the upper bound for the van der Corput property of the set of
perfect squares is $O((\log n)^{-1/3})$, giving an answer to a problem
considered by Ruzsa and Montgomery. We do it by constructing non-negative
valued, normed trigonometric polynomials with spectrum in the set of perfect
squares not exceeding $n$, and a small free coefficient $a_{0}=O((\log
n)^{-1/3})$.
\end{abstract}

\maketitle

\section{Introduction}

We say that a set $D$ of integers is a \textit{Poincar\'{e}} (\textit{%
recurrent, }or\textit{\ intersective}) set, if for any set $A$ of integers
with non-negative upper density 
\begin{equation*}
\rho (A)=\lim \sup_{n\rightarrow \infty }|A\cap \lbrack 1,n]|/n>0,
\end{equation*}%
its difference set $A-A$ contains an element of $D$. There is also an
equivalent ergodic theoretical characterization of the Poincar\'{e} property
(\cite{Furstenberg:77}). Furstenberg and S\'{a}rk\"{o}zy proved
independently that the sets of squares, sets of integer values of
polynomials with integer coefficients such that $P(0)=0$ and sets of shifted
primes $p-1$ and $p+1$ are Poincar\'{e} sets (\cite{Furstenberg:77}, \cite%
{Sarkozy:78a}, \cite{Sarkozy:78b}).

Given any set of integers $D$, one can define the function $\alpha :%
\boldsymbol{N}\rightarrow \lbrack 0,1]$ as $\alpha (n)=\sup \rho (A)$, where 
$A$ goes over all sets of integers whose difference set does not contain an
element of $D\cap \lbrack 1,n]$ (equivalent definitions of $\alpha $ can be
found in \cite{Ruzsa:84a}). One can then show that $D$ is Poincar\'{e} if
and only if%
\begin{equation*}
\lim_{n\rightarrow \infty }\alpha (n)=0\text{.}
\end{equation*}%
Upper bounds on the function $\alpha $ for the Poincar\'{e} sets mentioned
above have been obtained by various authors (\cite{Green:02}, \cite%
{Lucier:07}, \cite{Lucier:08}, \cite{Pintz:88}, \cite{Ruzsa:08}, \cite%
{Sarkozy:78a}, \cite{Sarkozy:78b}, \cite{Slijepcevic:03}), but even in the
simplest example of the set of squares, there is a huge gap between the best
upper and lower bounds for $\alpha $.

Kamae and Mend\`{e}z France introduced in \cite{Kamae:77} a closely related
notion of \textit{van der Corput} (or \textit{correlative}) sets, namely
sets of integers $D$ such that, given a real sequence $(x_{n})_{n\in N}$, if
all the sequences $(x_{n+d}-x_{n})_{n\in N}$, $d\in D$, are uniformly
distributed $\func{mod}1$, then the sequence $(x_{n})_{n\in N}$ is itself
uniformly distributed $\func{mod}1$\ (characterizations of the van der
Corput property are recalled in Section 2). Kamae and Mend\`{e}z France also
showed that van der Corput sets are Poincar\'{e} sets, and that all the
examples mentioned above are van der Corput sets.

Ruzsa introduced a function $\gamma :\boldsymbol{N}\rightarrow \lbrack 0,1]$
which quantifies the van der Corput property of a given set and gave several
characterizations of $\gamma $ (\cite{Montgomery:94}, \cite{Ruzsa:84a}).
Analogously as above, a set $D$ is a van der Corput set if and only if 
\begin{equation*}
\lim_{n\rightarrow \infty }\gamma (n)=0\text{.}
\end{equation*}

Ruzsa also showed that $\alpha \leq \gamma $. Ruzsa and Montgomery set a
problem of finding any upper bound for the function $\gamma $ for any
non-trivial van der Corput set, and in particular to find an upper bound for
the function $\gamma $ associated to the set of perfect squares (\cite%
{Montgomery:94}, unsolved problem 3; \cite{Ruzsa:84a}). They also
demonstrated that knowledge of upper bounds on the function $\gamma $ would
be useful, as $\gamma $ has various characterizations related to uniform
distribution and other properties of a set of integers.

In this paper we prove that for the set of squares, $\gamma (n)=O((\log
n)^{-1/3})$, and develop a technique which can likely be applied to other
van der Corput sets satisfying Kamae and Mend\`{e}z France condition (\cite%
{Kamae:77}, \S 3). We note that I. Ruzsa in \cite{Ruzsa:81} announced the
result that for the set of squares, $\gamma (n)=O((\log n)^{-1/2})$, but the
proof was never published.

It is important to emphasize that the gap between functions $\alpha $ and $%
\gamma $ can be arbitrarily large in relative terms. This was shown by
Bourgain, who constructed a set $D$ such that $\lim_{n\rightarrow \infty
}\alpha (n)=0$, while $\gamma (n)$ is bounded away from zero (\cite%
{Bourgain:87}). We argue in Section 2 that it is very difficult to obtain
van der Corput bounds for perfect squares better than $O((\log n)^{-1})$. We
also state the main result precisely in\ Section 2. In Section 3 we prove
the main result, postponing two key technical steps to Sections 4, 5. In
Section 6 we discuss an application of the main result to positive definite
functions vanishing on squares.

\section{Definitions and the main result}

We first introduce the notation, mostly following \cite{Montgomery:94}. If $%
D $ is a set of integers, then $D_{n}=D\cap \{1,...,n\}$. We denote by $%
\mathcal{T}(D)$ the set of all cosine polynomials%
\begin{equation}
T(x)=a_{0}+\tsum_{d\in D_{n}}a_{d}\cos (2\pi dx)\text{,}  \label{r:cosine}
\end{equation}%
$T(0)=1$, $T(x)\geq 0$ for all $x$, where $n$ is any integer and $%
a_{0},a_{d} $ are real numbers (i.e. $T$ is a normed non-negative valued
cosine polynomial with the spectrum in $D\cup \{0\}$).

Let $\mu $ be a Borel probability measure on the 1-torus $\boldsymbol{T}$
(parametrized with $[0,1)$). For $k\in \boldsymbol{Z}$ we define the Fourier
coefficients of $\mu $ to be the numbers%
\begin{equation*}
\widehat{\mu }(k)=\int \exp (-2\pi i\cdot kx)d\mu (x).
\end{equation*}

Let $\mathcal{M}(D)$ be the set of all probability measures on $\boldsymbol{T%
}$ such that $\widehat{\mu }(k)\not=0$ only when $|k|\not\in D$.

The following characterization of van der Corput sets is due to Kamae,\ Mend%
\`{e}z France and Ruzsa (\cite{Kamae:77}, \cite{Montgomery:94}, \cite%
{Ruzsa:84a}):

\begin{theorem}
\label{t:mainr}A subset $D$ of $\boldsymbol{N}$ is a van der Corput set if
and only if any of the following equivalent conditions hold:

(i)\ $\sup_{\mu \in \mathcal{M}(D)}\mu (\{0\})=0$,

(ii)\ $\inf_{T\in \mathcal{T}(D)}a_{0}=0$.
\end{theorem}

We can associate to a set $D$ two functions which describe how rapidly $D$
is becoming a van der Corput set:%
\begin{eqnarray}
\delta (n) &=&\sup_{\mu \in \mathcal{M}(D_{n})}\mu (\{0\}),  \label{r:vdc1}
\\
\gamma (n) &=&\inf_{T\in \mathcal{T}(D_{n})}a_{0}.  \label{r:vdc2}
\end{eqnarray}

Theorem \ref{t:mainr} now implies that a set is van der Corput if and only
if $\delta (n)\rightarrow 0$ as $n\rightarrow \infty $, or equivalently $%
\gamma (n)\rightarrow 0$ as $n\rightarrow \infty $. Ruzsa and Montgomery (%
\cite{Montgomery:94}, \cite{Ruzsa84b})\ proved the following result:%
\begin{equation*}
\gamma (n)=\delta (n)\text{.}
\end{equation*}

As was already noted in the introduction, the function $\gamma $ also
quantifies uniform distribution properties of a set $D$ (see \cite%
{Montgomery:94} for an exposition of the results). The function $\gamma $ is
an upper bound for the function $\alpha $ related to the Poincar\'{e}
property (see introduction), and also likely related to ergodic theoretical
and other properties related to the van der Corput property (\cite%
{Bergelson:07} contains the most recent results).

We focus in this paper on finding an upper bound for the function $\gamma $
associated to the set of perfect squares $Q$. Our approach is constructive:\
given $\delta >0$, we explicitly construct a non-negative normed cosine
polynomial (\ref{r:cosine})\ with coefficients in $Q_{n}\cup \{0\}$ and $%
a_{0}=\delta $.

Constructing non-negative trigonometric polynomials with a sparse set of
non-zero coefficients is not an easy task. We denote as usual $e(x)=\exp
(2\pi ix)$, and note that the real part of%
\begin{equation}
S(x,M)=\frac{1}{M}\tsum_{k=1}^{M}e(k^{2}x)  \label{r:start}
\end{equation}%
is a normed cosine polynomial. Recall that classical Weyl estimates show
that, if $M\gg q$ and $|x-p/q|\leq 1/q^{2}$ for some rational $p/q$, $%
(p,q)=1 $, then%
\begin{equation}
\left\vert \frac{1}{M}\tsum_{k=1}^{M}e(k^{2}x)\right\vert =O\left(
q^{-1/2}\right)  \label{r:start2}
\end{equation}%
(We prove a sharper version of (\ref{r:start2}) in Section 4). This means
that for large $M$, the sum (\ref{r:start})\ is sufficiently small for all $%
x $ which can be approximated by a rational with a large denominator;\ we
only need to fix "small denominators". A natural approach would be to define%
\begin{equation}
T_{L,M}(x)=\frac{1}{M}\tsum_{k=1}^{M}\cos (2\pi L^{2}k^{2}x)\text{.}
\label{r:start3}
\end{equation}

Then $T_{L,M}(x)$ is a cosine polynomial with non-zero coefficients only at
perfect squares, such that it is close to $1$ for $x$ which can be
approximated well by $p/q$, $q|L^{2}$. One can then hope that one can find
appropriate normalized weights $w_{k}$ such that for all $x$, the polynomial%
\begin{equation}
T(x)=\tsum_{L=1}^{L_{\max }}w_{k}T_{L,M}(x)\geq -\delta \text{.}
\label{r:intro1}
\end{equation}

We show that we can choose weights $w_{k}$ so that the values of $T_{L,M}(x)$
for rational $x$ with small denominator cancel out. This is more difficult
than it may seem, and is discussed in detail in Section 5. We also overcome
the second difficulty of oscillatory behavior of $T(x)$ near rationals with
a small denominator (see Proposition \ref{p:main3}), and prove the following
main result:

\begin{theorem}
\label{t:main2}If $\gamma $ is the function (\ref{r:vdc2})\ associated to
the set of perfect squares $Q$, then $\gamma (n)=O((\log n)^{-1/3})$.
\end{theorem}

The key step in the proof is using a constant $L$ which is the smallest
common multiplier of all the numbers between $1$ and $O(1/\delta ^{2})$, and 
$n\geq L$. Lemma \ref{l:4} implies that $\log n\geq \log L=O(1/\delta ^{2})$%
, so by inserting $\gamma (n)=\delta $ one sees that the best bound which
can be obtained by pursuing that approach is $\gamma (n)=O((\log n)^{-1/2})$%
. It is very difficult to do better than that, as the Kamae and Mend\`{e}z
France criterion (\cite{Kamae:77}, \S 3), which is in our knowledge
essentially the only known method of proving the van der Corput property of
squares, also depends on showing that the sums (\ref{r:start3}) for $L=q!$
are small (more specifically, converge to $0$ as $M\rightarrow \infty $ and $%
x$ irrational).

We devote the rest of this section to comparing our result to other upper
and lower bounds. Incidentally, our bound is essentially the same as the
bound $\alpha (n)=O((\log n)^{-1/3+\varepsilon })$ obtained by S\'{a}rk\"{o}%
zy (\cite{Sarkozy:78a}). In \cite{Pintz:88} the authors showed that $\alpha
(n)=O((\log n)^{-c})$ for arbitrarily large $c$. We argue that the van der
Corput property of squares is quantitatively fundamentally different than
the Poincar\'{e} property, and that $c$ can not be arbitrarily large.

We can denote by $\mathcal{T}^{+}(D)$ the set of all trigonometric
polynomials (\ref{r:cosine}) with nonnegative coefficients, and define $%
\gamma ^{+}(n)$ as in (\ref{r:vdc2}), where the infimum goes over $\mathcal{T%
}^{+}(D_{n})$. Then clearly $\gamma (n)\leq \gamma ^{+}(n)$. The methods
developed in this paper actually enable constructing only polynomials with
non-negative coefficients, and result with bounds on $\gamma ^{+}(n)$. I.
Ruzsa proved that for the set of squares, $\gamma ^{+}(n)\gg (\log n)^{-1}$,
which suggests that achieving $\gamma (n)=O((\log n)^{-c})$ for arbitrarily
large $c$ would be technically very difficult.

\section{Construction of the trigonometric polynomial}

Recall the definition of $S(x,M)$ defined in (\ref{r:start}). We also
introduce the function 
\begin{equation}
S(x,L,M)=S(L^{2}x,M)=\frac{1}{M}\sum_{k=1}^{M}e(k^{2}L^{2}x),
\label{r:fung1}
\end{equation}%
and $S(x,L,M)$ is also normed, $S(1,L,M)=1$. The following estimate is
essential in our construction.

\begin{proposition}
\label{p:main1}If $L,M$ are integers, $x\in \lbrack 0,1]$ and $|x-p/q|\leq
\varepsilon $, $(p,q)=1$, then%
\begin{eqnarray}
|S(x,L,M)| &=&\vartheta _{L}(q)+O\left( \frac{\sqrt{\log q}}{\sqrt{M}}+\frac{%
\sqrt{q\log q}}{M}+L^{2}M^{2}\varepsilon \right) ,  \label{r:est} \\
\vartheta _{L}(q) &=&\left\{ 
\begin{array}{cc}
1, & q|L^{2}\text{,} \\ 
0, & q/(q,L^{2})\equiv 2(\func{mod}4), \\ 
r^{-1/2} & \text{otherwise,}%
\end{array}%
\right.  \label{r:est00}
\end{eqnarray}%
where $r=q/(q,2L^{2})$.\ Furthermore, if $q|L^{2}$, then%
\begin{equation}
S(x,L,M)=1+O(L^{2}M^{2}\varepsilon )\text{.}  \label{r:est1}
\end{equation}
\end{proposition}

We dedicate the next section to the proof of Proposition \ref{p:main1},
modifying well-known bounds on $S$ based on the Weyl exponential sum
methods. The key difference to what is common in the literature is the
attention we put in evaluating precisely the leading term, e.g. in the case $%
L=1$ typically only bounded in the form $O(q^{-1/2})$. This is essential to
achieve the optimal bound at the end. We also discuss in Section 4 why we
use the unweighted exponential sum $S$ rather than a weighted version.

To simplify working with (\ref{r:est}), we set%
\begin{eqnarray}
\tau _{L}(q) &=&\left\{ 
\begin{array}{cc}
1, & q|L^{2}\text{,} \\ 
0, & q/(q,L^{2})\equiv 2(\func{mod}4), \\ 
-r^{-1/2} & \text{otherwise,}%
\end{array}%
\right.  \label{r:deftau} \\
E_{L,M}(q,\varepsilon ) &=&\min \left\{ c_{1}\left( \frac{\sqrt{\log q}}{%
\sqrt{M}}+\frac{\sqrt{q\log q}}{M}+L^{2}M^{2}\varepsilon \right) ,2\right\} ,
\notag
\end{eqnarray}%
where $r=q/(q,2L^{2})\,\ $and the constant in $E_{L,M}(q,\varepsilon )$ is
the larger of the constants in the error terms in (\ref{r:est}) and (\ref%
{r:est1}). We can then rewrite (\ref{r:est}), (\ref{r:est1}) as%
\begin{equation}
\func{Re}S(x,L,M)\geq \tau _{L}(q)-E_{L,M}(q,\varepsilon )\text{,}
\label{r:est2}
\end{equation}%
where $|x-p/q|\leq \varepsilon $, $(p,q)=1$. We will say in the following
that functions \thinspace $\tau $ or $E$ have a certain property for each $%
x\in \lbrack 0,1]$, if for a given $x$ they have that property for some $%
(p,q)=1$, $\varepsilon \geq 0$, where $|x-p/q|\leq \varepsilon $.

Ideally, for a given $\delta >0$, we would like to choose constants $L$, $M$
large enough so that $\tau -E$ is for each $x$ bounded from below by $%
-\delta $. As this is not possible for either of the terms, we will need to
average over many '$L$' (the term $\tau $) and over many '$M$' (the term $E$%
) to achieve that. We start with the term $\tau $.

\begin{proposition}
\label{p:main2}Say $\delta >0$. There exist constants $\lambda >0\,$\ and $%
1=L_{0}\leq L_{1}\leq ...\leq L_{l}=L_{\max }$, $\Lambda
=\sum_{k=0}^{l}\lambda ^{k}$ such that for any integer $q>0$,%
\begin{equation}
\frac{1}{\Lambda }\sum_{k=0}^{l}\lambda ^{k}\tau _{L_{k}}(q)\geq -\delta /2,
\label{r:rr}
\end{equation}%
and $L_{\max }=O(\exp c_{2}(1/\delta )^{2})$.
\end{proposition}

We dedicate the entire Section 5\ to the proof of Proposition \ref{p:main3},
as it consists of several steps somewhat combinatorial in character.

We now focus on the error term $E$.

\begin{proposition}
\label{p:main3}Say $\delta >0$ is small enough and $L\geq \exp (1/\delta )$.
Given\ any $x\in \lbrack 0,1]$, there exist constants $1\leq M_{1}\leq
...\leq M_{m}=M_{\max }$ depending only on $L,\delta $ and
constants\thinspace\ $p_{k},q_{k},\varepsilon _{k}$, $k=1,...,m$, where $%
(p_{k},q_{k})=1$ and $\varepsilon _{k}=|x-p_{k}/q_{k}|\,$, such that%
\begin{equation}
\frac{1}{m}\sum_{k=1}^{m}E_{L,M_{k}}(q_{k},\varepsilon _{k})\leq \delta /2%
\text{,}  \label{r:r3}
\end{equation}%
and $M_{\max }=O(L^{c_{3}\cdot 1/\delta }).$
\end{proposition}

\begin{proof}
Choose $m$ so that $8/\delta \leq m\leq 9/\delta $. We set \thinspace $%
M_{k}=L^{2(m+k)}$,$\,R_{k}\,\,=L^{4(m+k)}$, $\,k=1,...,m$. For a given $x\in
\lbrack 0,1]$, Let $p_{k}^{\prime }/q_{k}^{\prime }$, $(p_{k}^{\prime
},q_{k}^{\prime })=1$, be the sequence of Dirichlet's approximations of $x$,
i.e. the rationals such that $1\leq q_{k}^{\prime }\leq R_{k}$ and%
\begin{equation*}
\left\vert x-p_{k}^{\prime }/q_{k}^{\prime }\right\vert \leq
1/(q_{k}^{\prime }R_{k})\text{.}
\end{equation*}

We can also assume without loss of generality that $q_{k}^{\prime }$ is an
increasing sequence. Now, let $n$ be the largest index such that $%
q_{n}^{\prime }\leq L^{4m}$ ($n$ can also be $0$)\ We define $%
p_{k}/q_{k}=p_{n}^{\prime }/q_{n}^{\prime }$ for $k\leq n$, $%
p_{k}/q_{k}=p_{n+1}^{\prime }/q_{n+1}^{\prime }$ for $k\geq n+1$, and $%
\varepsilon _{k}=|x-p_{k}/q_{k}|$. We note that for $\delta $ small enough
(independent of $L$), $\log q_{k}\leq \log R_{m}\leq L^{1/2}$.

In the case $k\leq n-1$, using $\varepsilon _{k}\leq 1/R_{n}$ and $q_{k}\leq
L^{4m}$, we get%
\begin{equation*}
E_{L,M_{k}}(q_{k},\varepsilon _{k})\leq
L^{1/2}/M_{k}^{1/2}+L^{2m}L^{1/2}/M_{k}+L^{2}M_{k}^{2}/R_{n}\leq 3L^{-1/2}.
\end{equation*}

In the case $k\geq n+2$, using $\varepsilon _{k}\leq 1/(q_{n+1}R_{n+1})\leq
1/(L^{4m}R_{n+1})$ and $q_{k}\leq R_{n+1}$ we get%
\begin{equation*}
E_{L,M_{k}}(q_{k},\varepsilon _{k})\leq
L^{1/2}/M_{k}^{1/2}+R_{n+1}^{1/2}L^{1/2}/M_{k}+L^{2}M_{k}^{2}/(L^{4m}R_{n+1})\leq 3L^{-1/2}.
\end{equation*}

We conclude that for $\delta $ small enough (independent of $L$), for all $k$
except $k=n,n+1$, $E_{L,M_{k}}(q_{k},\varepsilon _{k})\leq \delta /4\,$\
holds. As for all $k$, $E_{L,M_{k}}(q_{k},\varepsilon _{k})\leq 2$ and $%
m\geq 8/\delta $, we easily obtain (\ref{r:r3}). Finally, $M_{\max
}=L^{4m}=O(L^{c_{3}\cdot 1/\delta })$ with $c_{3}=36$.
\end{proof}

We now complete the proof of Theorem \ref{r:vdc2}. Say $\delta >0$ is given.
We construct the cosine polynomial%
\begin{equation}
T(x)=\frac{1}{m\Lambda }\sum_{j=1}^{l}\sum_{k=1}^{m}\lambda ^{j}\func{Re}%
S(x,L_{j},M_{k}),  \label{r:cos}
\end{equation}%
where the constants $l,L_{1},...,L_{l}=L_{\max }$ are as constructed in
Proposition \ref{t:main2} and the constants $m,M_{1},...,M_{m}=M_{\max }$
are as constructed in Proposition \ref{p:main3} by choosing $L=$ $L_{\max }$%
. Using (\ref{r:est2}), (\ref{r:rr}), (\ref{r:r3}) and the fact that $%
E_{L,M} $ is non-decreasing in $L$, we obtain for each $x\in \lbrack 0,1]$%
\begin{equation*}
T(x)\geq \frac{1}{m\Lambda }\sum_{j=1}^{l}\sum_{k=1}^{m}\lambda ^{j}\left(
\tau _{L_{j}}(q_{k})-E_{L_{\max },M_{k}}(q_{k},\varepsilon _{k})\right) \geq
-\delta /2-\delta /2=-\delta \text{.}
\end{equation*}

The polynomial (\ref{r:cos}) is normed, has non-zero coefficients only at
perfect squares, and the largest non-zero coefficient is at $n=M_{\max
}^{2}L_{\max }^{2}=O(\exp (2c_{2}c_{3}(1/\delta )^{3})$, hence $\delta
=O((\log n)^{-1/3})$.

\section{Exponential sum estimates}

To prove Proposition \ref{p:main1}, we will here adapt classical upper
bounds on $S(x,M)$ based on the Weyl's method, following mostly the approach
and notation from \cite{Montgomery:94}, Section 3. As was mentioned earlier,
we do the adaptation to evaluate precisely the leading term below. Recall
the definition of $\vartheta _{L}(q)$ in (\ref{r:est00}), and then%
\begin{equation*}
\vartheta _{1}(q)=\left\{ 
\begin{array}{cc}
1, & q=1, \\ 
0, & q\equiv 2(\func{mod}4), \\ 
r^{-1/2} & \text{otherwise,}%
\end{array}%
\right.
\end{equation*}%
where $r=q/(q,2)$.

\begin{proposition}
\label{p:main10}If $p/q$ is a rational, $(p,q)=1$, then 
\begin{equation*}
|S(p/q,M)|=\vartheta _{1}(q)+O\left( \sqrt{\log q}/\sqrt{M}+\sqrt{q\log q}%
/M\right) \text{.}
\end{equation*}
\end{proposition}

\begin{proof}
Say $T=M^{2}|S|^{2}$, and then by substituting $h=k-j$ we see that 
\begin{eqnarray}
T
&=&\sum_{k,j=1}^{M}e((k^{2}-j^{2})p/q)=\sum_{h=1-M}^{M-1}%
\sum_{k=1}^{M-|h|}e(h^{2}p/q)e(2khp/q)=  \notag \\
&=&\sum_{q|2h}(M-|h|)e(h^{2}p/q)+\sum_{\rceil
q|2h}e(h^{2}p/q)\sum_{k=1}^{M-|h|}e(2khp/q),  \label{r:hp0}
\end{eqnarray}%
where $h$ in both sums in the second row goes from $1-M$ to $M-1$.

We first estimate the right-hand sum in (\ref{r:hp0}). If $q|2h$ does not
hold, 
\begin{equation}
\left\vert \sum_{k=1}^{M-|h|}e(2khp/q)\right\vert \leq 1/(2\left\Vert
2hp/q\right\Vert ),  \label{r:hp2}
\end{equation}%
where $\left\Vert x\right\Vert $ denotes the distance from $x$ to the
nearest integer. Choose a segment of variables $h$ which are not multipliers
of $r$ and of length $r-1$, and then we deduce that 
\begin{equation}
\sum_{h=kr+1}^{(k+1)r-1}\frac{1}{2\left\Vert 2hp/q\right\Vert }=O(q\log q)
\label{r:hp3}
\end{equation}%
(see e.g. \cite{Montgomery:94}, p.40 for details of evaluating (\ref{r:hp2})
and (\ref{r:hp3})). As there are at most $2M/r+2\leq 4M/q+2$ such segments,\
that, (\ref{r:hp2}) and (\ref{r:hp3}) imply that the absolute value of the
right-hand sum in (\ref{r:hp0}) is at most 
\begin{equation}
O(M\log q+q\log q).  \label{r:hp6}
\end{equation}

To evaluate the left-hand sum in (\ref{r:hp0}), we discuss two cases
depending on the remainder of $q$ $\func{mod}4$.

If $q\equiv 0,1,$ or $3(\func{mod}4)$, then $q|2h$ if and only if $q|h^{2}$,
and then $e(h^{2}p/q)=1$. If we set $r=q/(q,2)$, then $q|2h$ if and only if $%
r|h$ and the left-hand sum in (\ref{r:hp0}) becomes%
\begin{equation}
\sum_{r|h}(M-|h|)=2\sum_{k=1}^{M/r+1}(M-rk)+O(M)=M^{2}/r+O(M)+O(q)\text{.}
\label{r:hp1}
\end{equation}

Summing (\ref{r:hp6}) and (\ref{r:hp1})\ we deduce that $T=M^{2}/r+O(M\log
q+q\log q)$, which completes the proof in these cases.

If $q\equiv 2(\func{mod}4)$, then for $q|2h$, $e(h^{2}p/q)$ alternates
between $\pm 1$. Again $q|2h$ if and only if $r|h$ and the left-hand sum in (%
\ref{r:hp0}) becomes%
\begin{equation}
\sum_{r|h}(M-|h|)e(h^{2}p/q)=2\sum_{k=1}^{M/r+1}(M-rk)(-1)^{k}+O(M)=O(M)+O(q)%
\text{.}  \label{r:hp7}
\end{equation}

Summing (\ref{r:hp6}) and (\ref{r:hp7}) we obtain $T=O(M\log q+q\log q)$
which completes the proof if $q\equiv 2(\func{mod}4)$.
\end{proof}

We now complete the proof of Proposition \ref{p:main1}. Using Proposition %
\ref{p:main10}, relation $S(p/q,L,M)=S(L^{2}p/q,M)$ and the fact that the
error term is non-decreasing in $q$, we easily deduce that%
\begin{equation}
|S(p/q,L,M)|=\vartheta _{L}(q)+O\left( \sqrt{\log q}/\sqrt{M}+\sqrt{q\log q}%
/M\right) \text{.}  \label{r:hp4}
\end{equation}

Now say $|x-p/q|\leq \varepsilon $. As for any $k$ between $1$ and $M$,%
\begin{equation*}
\left\vert e(k^{2}L^{2}x)-e(k^{2}L^{2}p/q)\right\vert \leq 2\pi
k^{2}L^{2}|x-p/q|\leq 2\pi M^{2}L^{2}\varepsilon \text{,}
\end{equation*}%
We deduce that%
\begin{equation}
\left\vert S(x,L,M)-S(p/q,L,M)\right\vert =O(M^{2}L^{2}\varepsilon )\text{.}
\label{r:hp5}
\end{equation}

Combining (\ref{r:hp4}) and (\ref{r:hp5}) we obtain the first part of
Proposition \ref{p:main1}. We note that for any integer $n$, $S(n,M)=1$,
hence if $q|L^{2}$, $S(p/q,L,M)=1$. Combining that and (\ref{r:hp5}) we
obtain the second part of Proposition \ref{p:main1}.

The error term $M^{2}L^{2}\varepsilon $ above is not too good. We would like
to replace the exponential sum $S$ with a weighted exponential sum such that
an analogue of Proposition \ref{p:main1} holds with a better error term,
that means an error term such that the exponent on $M$ is less than twice
the exponent on $\varepsilon $. In that case, averaging over "$M$" and
Proposition \ref{p:main2} would not be required, and the bound in Theorem %
\ref{t:main2} would be improved to $O((\log n)^{-1/2})$. We dedicate the
rest of this section to discussing why two possible approaches do not
achieve that. One approach is choosing weights which simulate Dirichlet's
kernel as in \cite{Pintz:88}, and the other is simulating Fej\'{e}r's kernel.

\textit{Dirichlet's kernel. }The authors in \cite{Pintz:88} worked with
weighted exponential sums, and simulated normed Dirichlet's kernel%
\begin{equation*}
D_{M}(x)=\frac{1}{2M+1}\sum_{k=-M}^{M}e(kx)\text{.}
\end{equation*}%
Instead of $S(x,M)$ they defined the weighted sum $T(x,M)\,$approximating $%
D_{M}(x)$ as%
\begin{equation*}
T(x,M)=\frac{1}{M^{\prime }}\sum_{k=1}^{M}2ke(k^{2}x)\text{,}
\end{equation*}%
where $M^{\prime }$ is chosen so that $T(1,M)=1$. If $\left\vert
x-p/q\right\vert =\varepsilon $, then%
\begin{equation*}
T(x,M)=S(p/q,q)T(\varepsilon ,M)+O((q\log q)^{1/2}(1/M+M\varepsilon ))
\end{equation*}%
(\cite{Pintz:88}, relation (8)). As $T(\varepsilon ,M)=1+O(M^{2}\varepsilon
) $, the error term is essentially the same as in Proposition \ref{p:main1}.
(The authors in \cite{Pintz:88} also use the fact that $T(\varepsilon ,M)$
is close to $0$ when $\varepsilon $ is small but not too small, which is
opposite to our needs. We would wish to bound $T(\varepsilon ,M)$ close to $%
1 $ for small $\varepsilon $).

\textit{Fej\'{e}r's kernel. }Following the idea of I. Ruzsa, one can choose
weights to simulate the normed Fej\'{e}r's kernel%
\begin{equation*}
\Delta _{M}(x)=\frac{1}{M}\sum_{k=-M}^{M}\left( 1-\frac{|k|}{M}\right) e(kx)=%
\frac{1}{M^{2}}\left( \frac{\sin \pi Mx}{\sin \pi x}\right) ^{2}\text{,}
\end{equation*}%
with the purpose to dampen the oscillations of $S$ at integers and rational
numbers with small denominator. Instead of $S(x,M)$ we can define%
\begin{equation*}
V(x,M)=\frac{1}{M^{^{\prime \prime }}}\sum_{k=1}^{M}k\left( 1-\frac{k^{2}}{%
M^{2}}\right) e(k^{2}x)
\end{equation*}%
where $M^{^{\prime \prime }}$ is chosen so that $V(x,M)=1$. One can then
show that, if $\left\vert x-p/q\right\vert =\varepsilon $,%
\begin{equation*}
V(x,M)=S(p/q,q)\Delta _{M^{2}}(\varepsilon )+O(Mq\varepsilon )\text{.}
\end{equation*}%
As for small $\varepsilon $, $\Delta _{k}(\varepsilon )=1+O(k^{2}\varepsilon
^{2})$, we get $\Delta _{M^{2}}(\varepsilon )=1+O(M^{4}\varepsilon ^{2})$.
The error term $M^{4}\varepsilon ^{2}$ which replaces $M^{2}\varepsilon $ in
Proposition \ref{p:main1} in the case $L=1$ is better, but does not enable
us to improve the bound in\ Theorem \ref{t:main2}$.$

\section{Proof of Proposition \protect\ref{p:main3}}

\label{s:prop} \ Recall the definition of $\tau _{L}(q)$ in (\ref{r:deftau}%
). We prove here that we can find a linear combination of various $\tau _{L}$
so that its value for any $q$ is not smaller than $-\delta $ for a given
small $\delta >0$. The difficulty lies in the following. As was explained in
the introduction, choosing a very composite $L$ seems to be enough:\ say $n$
is greater than $1/\delta ^{2}$, and $L$ is the smallest common multiplier
of all numbers between $1$ and $n$. Then for most $q$, $\tau _{L}(q)\geq
-\delta $. Specifically, for numbers $q$ which divide $L^{2}$, $\tau
_{L}(q)=0$, and for numbers $q$ which have a prime factor larger than $n$, $%
\tau _{L}(q)\geq -\delta $. We, however, have no control over behavior of $%
\tau _{L}(q)$ for which $q/(q,2L^{2})$ is small (for example, multipliers of 
$2L^{2}$ with a small number, but also many other cases). This problem
arises for any $L$.

To resolve this and cancel out values of small $q/(q,2L^{2})$, we construct
an approximate geometric sequence of very composite "$L$'s". This idea is
coded in the Lemma \ref{l:1} below. For clarity, we write $\tau (L,q)$
instead of $\tau _{L}(q)$. Note that if $L$ and $q$ have only one common
prime number $p$ in their decompositions, then for $p=2$,%
\begin{equation*}
\tau (p^{j},p^{k})=\left\{ 
\begin{array}{cc}
1, & j-k/2\geq 0, \\ 
0, & j-k/2=-1/2, \\ 
-p^{j+1/2-k/2} & \text{otherwise,}%
\end{array}%
\right.
\end{equation*}%
and for $p\geq 3$, 
\begin{equation*}
\tau (p^{j},p^{k})=\left\{ 
\begin{array}{cc}
1, & j-k/2\geq 0, \\ 
-p^{j-k/2} & \text{otherwise.}%
\end{array}%
\right.
\end{equation*}

\begin{lemma}
\label{l:1}Say $p$ is a prime and $\mu $ a real number such that $1>\mu \geq
p^{-1/2}$. Then for any non-negative integers $n,k$,%
\begin{equation}
\sum_{j=0}^{n}\mu ^{j}\tau (p^{j},p^{k})\geq -\frac{1}{1-\mu }\mu ^{n+1}%
\text{.}  \label{r:maini}
\end{equation}
\end{lemma}

\begin{proof}
The case $k=0$ is trivial, so say $k\geq 1$. Denote the left-hand side of (%
\ref{r:maini}) with $A_{n}(p,k)$. Say $m$ is the largest index between $0$
and $n$ such that $m-k/2<0$, hence $m-k/2\leq -1/2$. We evaluate $A_{m}(p,k)$
in three cases. If $p\geq 3$, then using first $m-k/2\leq -1/2$ and then $%
-p^{-c}\geq -\mu ^{2c}$ for $c\geq 0$, we get%
\begin{eqnarray*}
A_{m}(p,k) &=&-\sum_{j=0}^{m}\mu ^{j}p^{j-k/2}\geq -\sum_{j=0}^{m}\mu
^{j}p^{j-m-1/2}\geq -\sum_{j=0}^{m}\mu ^{2m+1-j}= \\
&=&-\sum_{j=m+1}^{2m+1}\mu ^{j}\geq -\sum_{j=m+1}^{\infty }\mu ^{j}\text{.}
\end{eqnarray*}

Now assume $p=2$ and $m-k/2=-1/2$. As $\tau (p^{m},p^{k})=0$, similarly as
above we deduce that%
\begin{eqnarray*}
A_{m}(p,k) &=&-\sum_{j=0}^{m-1}\mu ^{j}p^{j+1/2-k/2}=-\sum_{j=0}^{m-1}\mu
^{j}p^{j-m}\geq -\sum_{j=0}^{m-1}\mu ^{2m-j}= \\
&=&-\sum_{j=m+1}^{2m}\mu ^{j}\geq -\sum_{j=m+1}^{\infty }\mu ^{j}\text{.}
\end{eqnarray*}

Finally, if $p=2$ and $m-k/2\leq -1$, repeating at the end the last couple
of steps as in the case $p\geq 3$ we obtain%
\begin{equation*}
A_{m}(p,k)=-\sum_{j=0}^{m}\mu ^{j}p^{j+1/2-k/2}\geq -\sum_{j=0}^{m}\mu
^{j}p^{j-m-1/2}\geq -\sum_{j=m+1}^{\infty }\mu ^{j}\text{.}
\end{equation*}

Inserting that in $A_{n}(p,k)$, we see that most of the terms cancel out:%
\begin{equation*}
A_{n}(p,k)=A_{m}(p,k)+\sum_{j=m+1}^{n}\mu ^{j}\geq -\sum_{j=n+1}^{\infty
}\mu ^{j}=-\frac{1}{1-\mu }\mu ^{n+1}\text{.}
\end{equation*}
\end{proof}

In the next step, we will fix the weights so that they do not depend on the
prime $p$. For clarity of the argument and notation, we write $\lambda
=2^{-1/2}$.

\begin{lemma}
\label{l:2}Say $l>0$ is an integer. For each prime number $p$ there exist
integers $0=d_{0}\leq d_{1}\leq ...\leq d_{l}$ such that for any integer $%
k\geq 0$,%
\begin{equation}
\sum_{j=0}^{l}\lambda ^{j}\tau (p^{d_{j}},p^{k})\geq -5\lambda ^{l}\text{,}
\label{r:hb}
\end{equation}

and 
\begin{equation}
p^{d_{l}}<2^{2l}.  \label{r:pmax}
\end{equation}
\end{lemma}

\begin{proof}
Let $e$ be an integer such that $2^{e+1}>p\geq 2^{e}$. Dividing $l$ with $e$
we get the quotient $f$ and the remainder $g$, $l=f\cdot e+g$. We define
coefficients $d_{j}$ so that $d_{0}=d_{1}=...=d_{e-1}=0$, and every $e$
coefficients we increase it by $1$ until we reach $f\cdot e$, and then $%
d_{f\cdot e}=...=d_{f\cdot e+g}=f$.

We denote the left-hand side of (\ref{r:hb}) with $B_{l}(p,k)$, and we set $%
\mu =2^{-e/2}=\lambda ^{e}$. We note that $\mu \geq p^{-1/2}$, apply Lemma %
\ref{l:1} and deduce that%
\begin{eqnarray}
B_{f\cdot e-1}(p,k) &=&\sum_{j=0}^{f-1}(1+\lambda +...+\lambda ^{e-1})\mu
^{j}\tau (p^{j},p^{k})\geq  \notag \\
&\geq &-\frac{1+\lambda +...+\lambda ^{e-1}}{1-\mu }\mu
^{f}=-\sum_{j=ef}^{\infty }\lambda ^{j}\text{.}  \label{r:b11}
\end{eqnarray}%
We analyse two cases. Say first $\tau (p^{f},p^{k})=1$, and then using (\ref%
{r:b11}) we get%
\begin{equation*}
B_{l}(p,k)=B_{f\cdot e-1}(p,k)+\sum_{j=ef}^{l}\lambda ^{j}\geq
-\sum_{j=l+1}^{\infty }\lambda ^{j}=-\frac{\lambda }{1-\lambda }\lambda ^{l}%
\text{.}
\end{equation*}%
Now say $\tau (p^{f},p^{k})\leq 0$, and then for all $j\leq f-1$, $\tau
(p^{j},p^{k})=p^{-1/2}\tau (p^{j},p^{k-1})$. The function $\tau
(p^{f},p^{k}) $ is always greater or equal than $-p^{-1/2}$. Using that, (%
\ref{r:b11}) and $-p^{-1/2}\geq -\mu =-\lambda ^{e}$, and finally $ef+e\geq
l+1$, we deduce that%
\begin{eqnarray*}
B_{l}(p,k) &=&p^{-1/2}B_{f\cdot e-1}(p,k-1)-\sum_{j=ef}^{l}\lambda
^{j}p^{-1/2}\geq \\
&\geq &-\sum_{j=ef+e}^{\infty }\lambda ^{j}-\sum_{j=ef+e}^{l+e}\lambda
^{j}\geq -\frac{2\lambda }{1-\lambda }\lambda ^{l}\text{.}
\end{eqnarray*}%
As $2\lambda /(1-\lambda )<5$, (\ref{r:hb}) holds. The relation (\ref{r:pmax}%
) follows from $p^{d_{l}}=p^{f}<2^{ef+f}\leq 2^{2ef}\leq 2^{2l}$.
\end{proof}

We now show why it is enough to study only primes.

\begin{lemma}
\label{l:3}Say $L_{1},...,L_{l}$ is a sequence of integers such that $%
L_{j}|L_{j+1}$ for all $j=1,...,l\,$. Then for each integer $q$, there
exists a prime $p$ such that for all $j$,%
\begin{equation}
\tau (L_{j},q)\geq \tau (p^{d_{j}},p^{k})\text{,}  \label{r:prim2}
\end{equation}%
where $p^{d_{j}}$, $p^{k}$ are factors in the prime decomposition of $L_{j}$%
, $q$ respectively.
\end{lemma}

\begin{proof}
Let $m+1$ be the smallest index such that $q|L_{m+1}^{2}$ (if there is no
such $m$, we set $m=l)$. If $m=0$, $q|L_{k}$ for all $k$, and we choose any
prime $p$ in the prime decomposition of $q$. Now say $1\leq m\leq l$, and
let $r=q/(q,L_{m}^{2})$. If $r\equiv 2(\func{mod}4)$, we set $p=2$,
otherwise we choose any prime $p$ in the prime decomposition of $r$. For $%
k\geq m+1$, both sides of (\ref{r:prim2}) are equal to 1. For $k\leq m$, it
is straightforward to check (\ref{r:prim2}).
\end{proof}

We now finally construct all variables in Proposition \ref{p:main3}. Choose $%
l$ so that $\delta /20\leq 2^{-l/2}\leq \delta /10$, and let $\lambda
=2^{-1/2}$, $\Lambda =\sum_{j=0}^{l}2^{-j/2}$. We set $n=2^{l}$. Let $%
2=p_{1}<p_{2}<...<p_{s}<n$ be all the prime numbers between $1$ and $n$, and
let $d_{j}^{i}$ be the exponents constructed in Lemma \ref{l:2}, associated
to the prime $p_{i}$, \thinspace $i=1,...,s$, $j=0,...,l$. We set%
\begin{equation*}
L_{j}=\prod_{i=1}^{s}p_{i}^{d_{j}^{i}}\text{.}
\end{equation*}%
Now applying Lemma \ref{l:3}\thinspace and then Lemma \ref{l:2} we deduce
that for any $q\in \boldsymbol{N}$,%
\begin{equation*}
\frac{1}{\Lambda }\sum_{j=0}^{l}\lambda ^{j}\tau (L_{j},q)\geq \frac{1}{%
\Lambda }\sum_{j=0}^{l}\lambda ^{j}\tau (p^{d_{j}},p^{k})\geq -5\cdot
2^{-l/2}\geq -\delta /2\text{.}
\end{equation*}

To complete the proof of Proposition \ref{p:main3}, we only need to estimate 
$L_{\max }=L_{l}$.

\begin{lemma}
\label{l:4}If $K$ is the smallest common multiplier of all numbers between $%
1 $ and $n$, then $K\leq \exp (c_{5}\cdot n)$, $c_{5}=1.04$.
\end{lemma}

\begin{proof}
This is \cite{Rosser:62}, Theorem 12.
\end{proof}

We now see that (\ref{r:pmax}) implies that $L_{\max }\leq K^{2}$, and by
Lemma \ref{l:4}, $K\leq \exp ^{c_{6}(1/\delta )^{2}}$, $c_{6}=1.04\cdot 400$.

\section{Positive definite functions vanishing of squares}

Now we discuss an application of Theorem \ref{t:main2} to positive definite
functions on $\boldsymbol{Z}/n\boldsymbol{Z}$ vanishing of squares.

We say that a number $\alpha \in \boldsymbol{Z}/n\boldsymbol{Z}$ is a
perfect square, if $\alpha \equiv \pm k^{2}(\func{mod}n)$ for some integer $%
k $, $k^{2}<n/2$. The fact that the set of squares is a Poincar\'{e} set
with estimates obtained in \cite{Pintz:88} can be interpreted as follows:

\begin{theorem}
\label{t:sfpss}\textbf{S\'{a}rk\"{o}zy, Furstenberg, Pintz, Steiger, Szem%
\'{e}redi. }If $A\subset \boldsymbol{Z}/n\boldsymbol{Z}$ such that $%
|A|/n\geq d_{4}(\log n)^{-d(n)}$, $d(n)=d_{5}\log \log \log \log n$, then $%
A-A$ contains a perfect square.
\end{theorem}

We now note that $A-A$ is not containing a perfect square if and only if the
function $1_{A}\ast 1_{-A}=1_{A}\ast 1_{A}^{\ast }$ vanishes on perfect
squares $\boldsymbol{Z}/n\boldsymbol{Z}$. The function $f=1_{A}\ast 1_{-A}$
is positive definite on $\boldsymbol{Z}/n\boldsymbol{Z}$ (i.e. all its
Fourier coefficients are real and non-negative, see \cite{Rudin:62}).

We can generalize the notion of density of a set to all non-zero complex
valued positive definite functions $f\in \boldsymbol{C}(\boldsymbol{Z}/n%
\boldsymbol{Z}\mathbf{)}$, and define it as%
\begin{equation*}
\rho (f)=\widehat{f}(0)/(nf(0))\text{,}
\end{equation*}%
One can easily check that $\rho (1_{A}\ast 1_{-A})=|A|/n$, so this is indeed
a natural generalization of the concept of density of a set.

\begin{proposition}
Say $f\in \boldsymbol{C}(\boldsymbol{Z}/n\boldsymbol{Z}\mathbf{)}$ is
non-zero, positive definite. Then $\rho (f)$ is well defined, $0\leq \rho
(f)\leq 1$. Furthermore, $\rho (f)=1$ if and only if $f$ is constant.
\end{proposition}

\begin{proof}
As for all positive definite functions, $f=0$ if and only if $f(0)=0$, $\rho
(f)$ is clearly well defined and non-negative. Calculating we get $\widehat{f%
}(0)=|\widehat{f}(0)|=|\tsum_{\alpha }f(\alpha )|\leq \tsum_{\alpha
}|f(\alpha )|\leq n|f(0)|=nf(0)$, hence $\rho (f)\leq 1$. The equality holds
in the inequalities above if the arguments and absolute values respectively
of $f(\alpha )$ are constant.
\end{proof}

We can now formulate the following strengthening of Theorem \ref{t:sfpss} as
a Corollary of Theorem \ref{t:main2}.

\begin{corollary}
\label{c:abc}Say $f\in \boldsymbol{C}(\boldsymbol{Z}/n\boldsymbol{Z}\mathbf{)%
}$ is non-zero, positive definite, such that $\rho (f)\geq d_{6}(\log
n)^{-1/3}$ for some constant $d_{6}$. Then $f$ can not vanish on all perfect
squares in $\boldsymbol{Z}/n\boldsymbol{Z}$.
\end{corollary}

\begin{proof}
Let $T(x)=\delta +\tsum_{{}}a_{d}\cos (2\pi dx)$, where sum goes over all $%
d\in Q_{n/2}$, be the non-negative cosine polynomial constructed in Theorem %
\ref{t:main2}, and say $f\in \boldsymbol{C}(\boldsymbol{Z}/n\boldsymbol{Z}%
\mathbf{)}$ is non-zero, positive definite, and $\rho (f)>\delta $. We
define a function $g\in \boldsymbol{C}(\boldsymbol{Z}/n\boldsymbol{Z}\mathbf{%
)}$ as%
\begin{equation*}
g(\alpha )=\frac{1}{2}\left\{ 
\begin{array}{cc}
a_{d}, & \alpha \equiv \pm d(\func{mod}n),d\in Q_{n/2}, \\ 
2\delta , & \alpha =0, \\ 
0 & \text{otherwise.}%
\end{array}%
\right.
\end{equation*}

Then by choice of $T(x)$, $g$ is positive definite, $g(0)=\delta $, $%
\widehat{g}(0)=1$. If $f$ vanishes on squares, we get%
\begin{equation*}
\delta \cdot f(0)=f\cdot g=\tsum_{\alpha }\frac{1}{n}\widehat{f}(\alpha )%
\widehat{g}(-\alpha )\geq \frac{1}{n}\widehat{f}(0)\widehat{g}(0)=\frac{1}{n}%
\widehat{f}(0)\text{,}
\end{equation*}%
hence $\rho (f)\leq \delta $ which is a contradiction (we used the notation $%
f\cdot g=\tsum_{\alpha }f(\alpha )g(\alpha )$ where $\tsum_{\alpha }$ stands
for $\tsum_{\alpha \in Z/nZ}$, a form of Parseval's identity on $Z/nZ$ and
positive definiteness of $f,g$).
\end{proof}

One can show that finding the functions $\alpha $, $\gamma $ is essentially
the same as finding the sharpest formulations of Theorem \ref{t:sfpss} and
Corollary \ref{c:abc}.

\begin{acknowledgement}
\ The author wishes to thank Professor Andrej Dujella for his help, and to
Professor Imre Z. Ruzsa for useful advice and encouragement.
\end{acknowledgement}

\end{document}